\documentclass[11pt]{amsart}

\usepackage{amsmath,amssymb}
\usepackage{xcolor}

\newtheoremstyle{mytheorem}
  {5pt}{5pt}{\itshape}{}{\bfseries}{}{.5em}
  {\thmname{#1}\thmnumber{ #2}\thmnote{ (#3)}.}
\theoremstyle{mytheorem}
\newtheorem{theorem}{Theorem}[section]
\newtheorem{lemma}[theorem]{Lemma}
\newtheorem{cor}[theorem]{Corollary}
\newtheorem{proposition}[theorem]{Proposition}

\theoremstyle{definition}
\newtheorem{definition}[theorem]{Definition}

\newtheoremstyle{boldremark}
  {3pt}{3pt}{\normalfont}{}{\bfseries}{.}{.5em}{}
\theoremstyle{boldremark}
\newtheorem{remark}[theorem]{Remark}

\newcommand{\cH}{\mathcal H}
\newcommand{\BH}{\mathcal B(\cH)}
\newcommand{\bC}{\mathbb C}
\newcommand{\bR}{\mathbb R}
\newcommand{\bN}{\mathbb N}
\newcommand{\Conj}{\operatorname{Conj}(\cH)}
\newcommand{\PH}{\mathcal P_1(\cH)}
\DeclareMathOperator{\Span}{span}
\DeclareMathOperator{\ran}{ran}
\DeclareMathOperator{\rank}{rank}
\newcommand{\SC}{\mathcal{S}_C}

\begin{document}

\title{Linear maps preserving families of $C$-symmetric operators}

\author[M. Anzai]{Mayu Anzai}
\address{Graduate School of Science and Technology,
Niigata University, Niigata 950-2181, Japan}
\email{mayu.anzai.math@gmail.com}

\author[S. Oi]{Shiho Oi}
\address{Department of Mathematics, Faculty of Science,
Niigata University, Niigata 950-2181, Japan}
\email{shiho-oi@math.sc.niigata-u.ac.jp}

\thanks{The second author was supported in part by JSPS KAKENHI
Grant Number JP24K06754.}
\date{}

\begin{abstract}
Let $\cH$ be a separable complex Hilbert space with $\dim\cH\ge3$.
We characterize the bounded bijective complex-linear maps $T$ on $\BH$
for which there exists a bijection $\psi$ of the set of conjugations,
preserving commutativity in both directions, such that $T$ maps the space
of $C$-symmetric operators onto the space of $\psi(C)$-symmetric operators
for every conjugation $C$.
We show that this condition is equivalent to the existence of a bijection $\varphi$ of the set of all orthonormal bases such that $T$ maps the set of all $(e_n)$-diagonal operators onto the set of all $\varphi((e_n))$-diagonal operators.
We also characterize the corresponding preservers in
dimension two using Pauli coordinates, without assuming the commutativity condition on $\psi$.
\end{abstract}

\maketitle

\section{Introduction}

Throughout this paper, $\cH$ denotes a separable complex Hilbert space,
and $\BH$ denotes the algebra of all bounded complex-linear operators
on $\cH$. We write $\BH_{sa}$ for its self-adjoint part and $I$ for
the identity operator. The inner product is linear in its first variable.
Unless stated otherwise, linear maps are complex-linear, and projections
on $\cH$ are complex-linear orthogonal projections. 

\begin{definition}
A map $C:\cH\to\cH$ is a conjugation if
\begin{enumerate}
  \item[(i)] $C$ is antilinear, i.e. $C(\alpha x+y)=\overline{\alpha} C x+ Cy$ for $x,y\in\cH$ and $\alpha\in\bC$,
  \item[(ii)] $C$ is invertible with $C^{-1}=C$, and
  \item[(iii)] $\langle Cx, Cy\rangle=\langle y,x\rangle$ for all $x, y\in\cH$.
\end{enumerate}
We denote the set of all conjugations on $\cH$ by $\Conj$.
For an orthonormal basis $(e_i)$, the unique conjugation satisfying
$Ce_i=e_i$ for every $i$ is called the conjugation associated with $(e_i)$.
\end{definition}

\begin{definition}
An operator $A\in\BH$ is $C$-symmetric if $CAC=A^*$.
It is complex symmetric if it is $C$-symmetric for some $C\in\Conj$.
For each $C\in\Conj$, set
\[
\mathcal S_C=\{A\in\BH\mid CAC=A^*\}.
\]
\end{definition}

Each $\mathcal S_C$ is a complex-linear subspace of $\BH$ closed under
taking adjoints. The set of all complex symmetric operators is
$\bigcup_{C\in\Conj}\mathcal S_C$.
Garcia and Putinar initiated the study of complex symmetric operators
and showed that this class contains many important classes of operators
(see \cite{GP}). In this paper, we study the whole family
$\{\mathcal S_C\mid C\in \Conj\}$ rather than a single
space $\mathcal S_C$. The intersection structure of this family
contains information about projections and diagonal operator
algebras.

Write $N=\dim\cH$, using $N=\infty$ in the infinite-dimensional case,
and set
\[
\mathcal I_N=
\begin{cases}
\{1,\ldots,N\},&N<\infty,\\
\bN,&N=\infty.
\end{cases}
\]
Let $\mathcal E(\cH)$ be the set of all ordered orthonormal bases
$\mathbf e=(e_i)_{i\in\mathcal I_N}$ of $\cH$.
Here, we regard an orthonormal basis as an ordered sequence. Thus, $(e_1, e_2, \cdots)$ and $(e_2, e_1, \cdots)$ are regarded as distinct elements of $\mathcal{E}(\cH)$. 
For $\mathbf e=(e_i)\in\mathcal E(\cH)$, define the diagonal algebras with respect to $\mathbf e$ by
\[
\mathcal D_{\mathbf e}
=\{A\in\BH\mid Ae_i\in\bC e_i\text{ for every }i\in\mathcal I_N\}.
\]
An operator in $\mathcal D_{\mathbf e}$ is called $\mathbf e$-diagonal;
an operator is called diagonal if it is $\mathbf e$-diagonal for some
$\mathbf e\in\mathcal E(\cH)$.

Linear preserver problems have been extensively studied on matrix
algebras $M_n(\bC)$ and on $\BH$.
In particular, linear maps preserving normal operators were characterized
in \cite{Brs}, and linear maps preserving diagonalizable matrices were
characterized in \cite{OS}. Ji, Liu, and Zhu studied linear maps preserving the class of all complex symmetric operators. Their problem concerns the preservation of 
$\bigcup_{C\in\Conj}\mathcal S_C$. As they observed, ``Indeed, there are too many linear maps on $\BH$ preserving complex symmetry'' \cite[p.~1145]{JLZ}. They therefore focused on several special linear maps, including similarity transformations, surjective linear isometries, and multiplication operators.
Indeed, when $\dim H=2$, every operator on $H$ is complex symmetric. Therefore,
$\bigcup_{C\in\operatorname{Conj}(H)}\mathcal S_C=\mathcal B(H)=M_2(\mathbb C)$. Thus, in dimension two, the condition that a map preserves the class of all complex symmetric operators gives no information about the map.
On the other hand, the individual spaces $\mathcal S_C$ are still different. Hence, we study maps preserving families of $C$-symmetric operators rather than only the union of its members.
As a related problem, Amara, Oudghiri, and Souilah \cite{AOS} characterized additive maps
that leave each $\mathcal S_C$ invariant, or equivalently, each
$\mathcal D_{\mathbf e}$ invariant. More precisely, they proved the following theorem.

\begin{theorem}[{\cite[Theorem 1]{AOS}}]\label{AOST}
Let $\cH$ be a separable complex Hilbert space with $\dim\cH\ge3$,
and let $T:\BH\to\BH$ be an additive map. The following are equivalent:
\begin{enumerate}
\item $T(\mathcal S_C)\subset\mathcal S_C,$ for every $C\in\Conj$.
\item $T(\mathcal D_{\mathbf e})\subset\mathcal D_{\mathbf e},$
for every $\mathbf e\in\mathcal E(\cH)$.
\item There exist $\alpha,\beta\in\bC$ and an additive functional
$f:\BH\to\bC$ such that
\[
T(A)=\alpha A+\beta A^*+f(A)I\qquad(A\in\BH).
\]
\end{enumerate}
\end{theorem}
In this paper, instead of requiring every space $\mathcal S_C$ to be mapped into itself, we allow $T$ to send it onto a possibly different member of the same family. More precisely, we study bounded bijective complex-linear maps $T: \BH \to \BH$ for which there exists a bijection 
\[
\psi:\Conj\to\Conj
\]
such that 
\[
T(\mathcal S_C)=\mathcal S_{\psi(C)}\quad\text{for every $C\in\Conj$}.
\]
To the best of our knowledge, this family-preserving problem has not been studied before. 
In the problem considered in~\cite{AOS}, every member of the family remains fixed. In our problem, the members may be permuted by a nontrivial map $\psi$. Thus, one of the main difficulties is to understand the map $\psi$.

This gives a meaningful preserver problem even when  $\dim \cH=2$. Using the Pauli matrices, we give a complete characterization of such maps.

\begin{definition}
The Pauli matrices are
\[
\sigma_1=\begin{pmatrix}0&1\\1&0\end{pmatrix},\qquad
\sigma_2=\begin{pmatrix}0&-i\\i&0\end{pmatrix},\qquad
\sigma_3=\begin{pmatrix}1&0\\0&-1\end{pmatrix}.
\]
For $x=(x_1,x_2,x_3)\in\bC^3$, write
$x\cdot\sigma=\sum_{j=1}^3x_j\sigma_j$.
Every $A\in M_2(\bC)$ has a unique representation
$A=aI+x\cdot\sigma$ with $a\in\bC$ and $x\in\bC^3$.
\end{definition}

We write $GL(3,\bR)$ for the invertible real $3\times3$ matrices.
Such a matrix also denotes its complex-linear extension to $\bC^3$.

\begin{theorem}\label{main0}
Let $\dim\cH=2$, and identify $\cH$ with $\bC^2$ by fixing an
orthonormal basis. Let $T:M_2(\bC)\to M_2(\bC)$ be a bijective
complex-linear map. The following are equivalent:
\begin{enumerate}
\item There exists a bijection $\psi:\Conj\to\Conj$ such that
\[
T(\mathcal S_C)=\mathcal S_{\psi(C)},\qquad \text{for every $C\in\Conj$.}
\]
\item There exists a bijection
$\varphi:\mathcal E(\cH)\to\mathcal E(\cH)$ such that
\[
T(\mathcal D_{\mathbf e})=\mathcal D_{\varphi(\mathbf e)},
\qquad \text{ for every $\mathbf e\in\mathcal E(\cH)$}.
\]
\item There exist nonzero scalars $c,\lambda\in\bC$, a complex-linear
functional $\ell$ on $\bC^3$, and a matrix $B\in GL(3,\bR)$ such that
\[
T(aI+x\cdot\sigma)=caI+\ell(x)I+\lambda(Bx)\cdot\sigma
\]
for every $a\in\bC$ and $x\in\bC^3$.
\end{enumerate}
\end{theorem}
When $\dim H\geq 3$, we make the additional assumption that $\psi$ preserves commutativity in both directions. Under this assumption, we obtain an explicit description of $T$ and an equivalent condition involving
diagonal operator algebras. 

\begin{definition}
    A bijection $\psi:\Conj\to\Conj$ is said to preserve commutativity
in both directions if
\[
C_1C_2=C_2C_1
\quad\Longleftrightarrow\quad
\psi(C_1)\psi(C_2)=\psi(C_2)\psi(C_1)
\]
for all $C_1,C_2\in\Conj$.
\end{definition}

\begin{theorem}\label{main}
Let $\cH$ be a separable complex Hilbert space with $\dim\cH\ge3$.
Let $T:\BH\to\BH$ be a bounded bijective complex-linear map.
The following are equivalent:
\begin{enumerate}
\item There exists a bijection $\psi:\Conj\to\Conj$ that preserves
commutativity in both directions and satisfies
\[
T(\mathcal S_C)=\mathcal S_{\psi(C)},\qquad \text{ for every $C\in\Conj$}.
\]
\item There exists a bijection
$\varphi:\mathcal E(\cH)\to\mathcal E(\cH)$ such that
\[
T(\mathcal D_{\mathbf e})=\mathcal D_{\varphi(\mathbf e)},
\qquad\text{ for every $\mathbf e\in\mathcal E(\cH)$}.
\]
\item There exist a nonzero scalar $c\in\bC$, a bounded complex-linear
functional $f$ on $\BH$, and a unitary or antiunitary operator $U$ on
$\cH$ such that
\[
T(A)=
\begin{cases}
cUAU^*+f(A)I,&\text{if }U\text{ is unitary},\\[2pt]
cUA^*U^*+f(A)I,&\text{if }U\text{ is antiunitary},
\end{cases}
\qquad A\in\BH.
\]
\end{enumerate}
\end{theorem}

\begin{remark}\label{invertibility}
In condition (3) of Theorem~\ref{main}, let $J(A)=UAU^*$ in the
unitary case and $J(A)=UA^*U^*$ in the antiunitary case.
Then $J$ is a complex-linear bijection with $J(I)=I$.
For a nonzero scalar $c$ and a bounded complex-linear functional $f$,
the map $T(A)=cJ(A)+f(A)I$ is bijective if and only if $c+f(I)\ne0$.
Necessity follows from $T(I)=(c+f(I))I$; sufficiency follows from
\[
T^{-1}(B)=\frac1cJ^{-1}(B)
-\frac{f(J^{-1}(B))}{c(c+f(I))}I.
\]
Thus this restriction on $c$ and $f$ is already implied by the
bijectivity assumption in Theorem~\ref{main}.
\end{remark}

The paper is organized as follows. 
In Section \ref{Preliminaries}, we describe some properties of the spaces $\mathcal S_C$ and relate them
to projections and diagonal algebras. 
Section \ref{section3} is devoted to the
two-dimensional case. In this case, the Pauli representation gives a one-to-one correspondence between the distinct spaces of the form $\mathcal S_C$ and the real two-dimensional planes in $\mathbb R^3$, through their traceless parts (Lemma~\ref{plane}). Every invertible real-linear map on $\mathbb R^3$ sends such planes to planes. Therefore, $\operatorname{GL}(3,\mathbb R)$, rather than only the orthogonal group, appears in the description of the maps. No additional assumption on $\psi$, such as the commutativity-preserving condition introduced below, is needed in this case. In Section \ref{section4}, we prove Theorem \ref{main}. First, using the commutativity-preserving bijection $\psi$ on the set of all conjugations, we construct a bijection between the corresponding sets of real orthogonal projections. This bijection also preserves commutativity in both directions. We use it to prove $(1) \Rightarrow (2)$. 
The implication $(2)\Rightarrow(3)$ can also be proved more briefly by applying the theorem of Bre\v{s}ar and \v{S}emrl on linear maps preserving normal operators \cite[Theorem~2]{Brs}. However, our proof using rank-one projections shows directly how the preservation of diagonal algebras leads to the unitary or antiunitary operator.

\section{Preliminaries}\label{Preliminaries}

Throughout this section, assume that $\dim\cH\ge2$.

\subsection{Notation and elementary facts}

Let $\mathbb T=\{\zeta\in\bC:|\zeta|=1\}$.
For a subset $S$ of a complex vector space, $\Span S$ and
$\Span_{\bR}S$ denote its complex and real linear spans, respectively.
For $x,y\in\cH$, define $x\otimes y\in\BH$ by
\[
(x\otimes y)z=\langle z,y\rangle x\qquad(z\in\cH).
\]
We denote the spectrum of $A\in\BH$ by $\sigma(A)$, and set
\[
\mathcal P(\cH)=\{P\in\BH:P^2=P=P^*\},
\]
\[
\mathcal P_1(\cH)=\{P\in\mathcal P(\cH):\rank P=1\}.
\]
Ranks of operators on $\cH$ are taken over $\bC$.

\begin{lemma}[{\cite[Lemma 1]{GP}}]\label{GP1}
For every $C\in\Conj$, there exists an orthonormal basis
$\mathbf e=(e_i)\in\mathcal E(\cH)$ such that $Ce_i=e_i$ for every $i$.
\end{lemma}

\begin{lemma}\label{I}
A self-adjoint operator $A\in\BH$ belongs to $\mathcal S_C$ for every
$C\in\Conj$ if and only if $A=cI$ for some $c\in\bR$.
\end{lemma}
\begin{proof}
We consider two cases: (i) $\dim \cH=2$ and (ii) $\dim \cH \ge 3$.\\
(i) Suppose $\dim \cH=2$. Assume that $A \in \SC$ for all conjugations $C$. Let $U$ be a unitary matrix such that 
\[
U^{*}AU=\begin{pmatrix}
\lambda_1 & 0  \\
0 & \lambda_2 
\end{pmatrix},
\]
 where $\sigma(A)=\{ \lambda_1, \lambda_2\}$. Let $\{(1,0) ,(0,1)\}$ be the standard basis of $\mathbb{C}^2$. 
A conjugation $J$ on $\mathbb{C}^2$ is defined by
\[
J(z_1, z_2)
=
(\overline{z_2}, \overline{z_1})
\]
for every $(z_1,z_2)\in\mathbb{C}^2$, where $z_1$ and $z_2$ are the coordinates with respect to the standard basis $\{(1,0) ,(0,1)\}$.  Then we get a conjugation $UJU^{*}$.  Since $A=UJU^{*}AUJU^{*}$, we get $U^{*}AU=JU^{*}AUJ$.  This implies 
\[
\begin{pmatrix}
\lambda_1 & 0 \\
0 & \lambda_2 
\end{pmatrix}=\begin{pmatrix}
\lambda_2 & 0 \\
0 & \lambda_1 
\end{pmatrix}.
\]
Thus we get $\lambda_1=\lambda_2$, which implies $A=cI$ for some $c \in \mathbb{R}$. \\

(ii) Suppose $\dim \cH \ge 3$. We assume that  $A \in \SC$ for all conjugations $C$.  Let $x \in \cH$ with $x \neq 0$. Then there are $\alpha \in \mathbb{C}$ and $y \in (\Span \{x\} )^{\perp} $ such that
\[
Ax= \alpha x + y.
\]
Suppose that $y \neq 0$. As $\dim \cH \ge 3$, there is $z \in (\Span \{x, y\} )^{\perp} $ with $\|z\|=1$. 
Let $C$ be a conjugation such that $Cx=x$ and $C\bigl(\frac{y}{\|y\|}\bigr)=z$. We have 
\[
CAx=ACx=Ax.
\]
Thus 
\[
\overline{\alpha} x + \|y\|z= C(Ax)=\alpha x + y, 
\]
which is a contradiction. Thus $y=0$. Since $A \in \BH_{sa}$, we obtain $\alpha \in \mathbb{R}$. 
We define $\alpha(x) \in \mathbb{R}$ for any $x \in \cH \setminus \{0\}$  by 
\[
Ax= \alpha(x)x.
\]
Let $x, y \in \cH \setminus \{0\}$. If $\dim \Span \{x,y\}=1$, then there is $\lambda \in \mathbb{C}$ such that $x=\lambda y$. We get 
\[
\alpha(x)x=Ax=A(\lambda y)=\lambda Ay=\lambda \alpha(y) y=\alpha(y) x.
\]
As $x \neq 0$, we obtain $\alpha(x)=\alpha(y)$. If $\dim \Span \{x,y\}=2$, there is a non-zero vector $z \in \Span \{x,y\} ^{\perp} \setminus \{0\}$ as $\dim \cH \ge 3$. Then there are orthogonal bases of $\cH$, $(x,z, \cdots)$ and $(y, z, \cdots)$. We have
\[
\alpha(x+z)(x+z)=A(x+z)=Ax+Az=\alpha(x)x+\alpha(z)z.
\] 
This gives
\[
\alpha(x)=\alpha(x+z)=\alpha(z).
\]
By a similar argument for $(y, z, \cdots)$, we have
\[
\alpha(y)=\alpha(y+z)=\alpha(z).
\]
Thus we get $\alpha(x)=\alpha(z)=\alpha(y)$. Hence $\alpha(x)$ is independent of $x$, so $A=\alpha I$ for some $\alpha\in\mathbb R$. The converse follows immediately.
\end{proof}

\begin{cor}\label{comp}
We have
\[
\bigcap_{C\in\Conj}\mathcal S_C=\bC I.
\]
\end{cor}
\begin{proof}
The inclusion from right to left is immediate.
If $A \in \bigcap_{C\in\Conj}\mathcal S_C$, then its self-adjoint parts
$(A+A^*)/2$ and $(A-A^*)/(2i)$ belong to $\bigcap_{C\in\Conj}\mathcal S_C$.
Lemma~\ref{I} shows that both are real scalar multiples of $I$,
and hence $A\in\bC I$.
\end{proof}

\begin{lemma}\label{dim1}
For $\Lambda\subset\Conj$, the intersection
$\bigcap_{C\in\Lambda}\mathcal S_C$ is one-dimensional if and only if
it equals $\bC I$.
\end{lemma}
\begin{proof}
We assume that $\dim \bigcap_{C \in \Lambda} \SC =1$. Then Lemma \ref{I} shows that $I \in \bigcap_{C \in \Lambda} \SC $. Since $\bigcap_{C \in \Lambda} \SC$ is a complex linear subspace of $B(\cH)$, we have $\Span \{I\} \subset \bigcap_{C \in \Lambda} \SC $. As $\dim \bigcap_{C \in \Lambda} \SC =1$, we obtain $\bigcap_{C \in \Lambda} \SC=\Span \{I\}$. The converse is clear.
\end{proof}

\begin{lemma}\label{projectionC}
Let $P \in \BH$ be a projection and $C \in \Conj$. Then 
\[
C(P(\cH))=P(\cH) \Longleftrightarrow P \in \SC.
\]
\end{lemma}
\begin{proof}
Assume that $C(P(\cH))=P(\cH)$. Let $x \in P(\cH)$ and  $y=Cx \in P(\cH)$. Therefore, $CPCx=CPy=Cy=x=Px$. For any $x \in P(\cH)^{\perp}$, we have
\[
\langle Cx, y \rangle=\langle Cy, x \rangle =0
\]
for any $y \in P(\cH)$. This implies that $Cx \in  P(\cH)^{\perp}$. Hence we get $CPCx=0=Px$. Thus,  we have
\[
CPC=P.
\]
To prove the converse implication, let $x \in P(\cH)$. We have $Cx=CPx=PCx \in P(\cH)$. Thus $C(P(\cH)) \subset P(\cH)$. It follows that $P(\cH)=C^2(P(\cH)) \subset CP(\cH)$. Hence we get $C(P(\cH)) = P(\cH)$.
\end{proof}

\begin{remark}\label{projC}
If $P \in \SC$, the restrictions of $C$ to $\ran P=P(\cH)$ and $\ker P=(I-P)(\cH)$ are
conjugations on these subspaces. Conversely, the direct sum of any
conjugation on $\ran P$ and any conjugation on $\ker P$ is a conjugation
on $\cH$ commuting with $P$.
\end{remark}

\subsection{Conjugations and diagonal algebras}

For $\mathbf e=(e_i)_{i\in\mathcal I_N}\in\mathcal E(\cH)$ and
$\rho:\mathcal I_N\to\{\pm1\}$, let $C_{\mathbf e}^{\rho}$ be the
conjugation determined by
\[
C_{\mathbf e}^{\rho}\Bigl(\sum_i z_i e_i\Bigr)
=\sum_i\overline{z_i}\rho(i)e_i.
\]
Write $C_{\mathbf e}$ for the conjugation corresponding to
$\rho(i)=1$ for every $i$, and define
\[
\Lambda_{\mathbf e}
=\{C_{\mathbf e}^{\rho} \mid \rho:\mathcal I_N\to\{\pm1\}\},
\qquad
\mathcal P_{\mathbf e}=\mathcal P(\cH)\cap\mathcal D_{\mathbf e}.
\]
Thus $\mathcal P_{\mathbf e}$ is the set of all $\mathbf e$-diagonal
orthogonal projections.

\begin{lemma}[{\cite[Lemma 1]{AOS}}]\label{L1}
For every $\mathbf e\in\mathcal E(\cH)$,
\[
\mathcal D_{\mathbf e}
=\bigcap_{C\in\Lambda_{\mathbf e}}\mathcal S_C.
\]
\end{lemma}

\begin{lemma}\label{proj1}
For $\mathbf e\in\mathcal E(\cH)$ and $C_0\in\Lambda_{\mathbf e}$,
\[
\Lambda_{\mathbf e}
=\{(I-2P)C_0\mid P\in\mathcal P_{\mathbf e}\}.
\]
\end{lemma}
\begin{proof}
Since $C_0\in\Lambda_{\mathbf e}$, there is a map
$\rho_0:\mathcal I_N\to\{\pm1\}$ such that $C_0e_i=\rho_0(i)e_i$.
First, let $P \in \mathcal P_{\mathbf e}$. Then there is a map
$s:\mathcal I_N\to\{0,1\}$ such that $Pe_i=s(i)e_i$. By a simple calculation, it is easy to check that $I-2P$ is self-adjoint and unitary and $PC_0=C_0P$. Thus $(I-2P)C_0$ is an antilinear isometry with $((I-2P)C_0)^2=I$.
Hence  $(I-2P)C_0 \in \Conj$. Moreover,
\[
(I-2P)C_0e_i=\bigl(1-2s(i)\bigr)\rho_0(i)e_i.
\]
Since $\bigl(1-2s(i)\bigr)\rho_0(i) \in\{\pm1\}$,
we obtain $(I-2P)C_0\in\Lambda_{\mathbf e}$.
Therefore,
\[
\left\{
(I-2P)C_0\mid P\in\mathcal P_{\mathbf e}
\right\}
\subset
\Lambda_{\mathbf e}.
\]

Conversely, let $C^{\rho}_{\mathbf e}\in\Lambda_{\mathbf e}$. Define $Pe_i=\frac{1-\rho(i)\rho_0(i)}{2}e_i$. Then $P \in \mathcal P_{\mathbf e}$. We have
\[
(I-2P)C_0e_n=\bigl(1-(1-\rho(n)\rho_0(n))\bigr)\rho_0(n)e_n= \rho(n)e_n=C^{\rho}_{\mathbf e}(e_n).
\]
Thus $C^{\rho}_{\mathbf e}=(I-2P)C_0$. 
This proves the reverse inclusion. 
\end{proof}

For projections $P_1,\ldots,P_m\in\mathcal P(\cH)$, set
\[
\Gamma(P_1,\ldots,P_m)
=\{C\in\Conj\mid CP_jC=P_j\text{ for }j=1,\ldots,m\}.
\]
For each nonzero $x\in\cH$, define
\[
P_x=\frac{x\otimes x}{\|x\|^2},\qquad
\Lambda_x=\Gamma(P_x)=\{C\in\Conj\mid Cx\in\mathbb T x\}.
\]
The last equality follows from Lemma~\ref{projectionC} and the
isometric property of conjugations.

\begin{lemma}\label{finite-blocks}
Let $P_1,\ldots,P_m$ be nonzero mutually orthogonal projections with
$\sum_{j=1}^mP_j=I$. Then
\[
\bigcap_{C\in\Gamma(P_1,\ldots,P_m)}\mathcal S_C
=\Span\{P_1,\ldots,P_m\}.
\]
\end{lemma}
\begin{proof}
Each $P_j$ belongs to the intersection by Lemma~\ref{projectionC}. Thus we get
\[
\Span\{P_1,\ldots,P_m\}
 \subset \bigcap_{C\in\Gamma(P_1,\ldots,P_m)}\mathcal S_C.
\]
To prove the reverse inclusion,  let $A \in \bigcap_{C\in\Gamma(P_1,\ldots,P_m)}\mathcal S_C$. Since $\SC$ is closed under the involution $*$, it contains both $(A+A^*)/2$ and $(A-A^*)/(2i)$.
It therefore suffices to prove the reverse inclusion under the assumption that $A$ is self-adjoint.  Put $\cH_j=P_j(\cH)$. Choose a conjugation on each $\cH_j$, and let
$C_0$ be their direct sum. For every $j$, we have
\[
C_0\in \Gamma(P_1,\ldots,P_m), \quad D_j=(I-2P_j)C_0 \in \Gamma(P_1,\ldots,P_m).
\]
Since $A=A^*$, it commutes with $C_0$ and $D_j$, and hence with
$D_jC_0=I-2P_j$. Thus $A$ is block diagonal for
$\cH=\bigoplus_{j=1}^m\cH_j$.

Every conjugation on $\cH_j$ extends to an element of
$\Gamma(P_1,\ldots,P_m)$ by taking its direct sum with conjugations
on the other blocks. Therefore the self-adjoint operator $A|_{\cH_j}$
commutes with every conjugation on $\cH_j$.
If $\dim\cH_j\ge2$, Lemma~\ref{I} shows that
$A|_{\cH_j}=a_jI_{\cH_j}$ for some $a_j\in\bR$; the same conclusion
is immediate when $\dim\cH_j=1$.
Hence $A=\sum_{j=1}^m a_jP_j$, as required.
\end{proof}

\begin{cor}\label{projection-intersections}
For $P\in\mathcal P(\cH)$,
\[
\bigcap_{C\in\Gamma(P)}\mathcal S_C=\Span\{I,P\}.
\]
If $P_1,P_2\in\mathcal P(\cH)$ satisfy $P_1P_2=0$, then
\[
\bigcap_{C\in\Gamma(P_1,P_2)}\mathcal S_C=\Span\{I,P_1,P_2\}.
\]
In particular, for every nonzero $x\in\cH$,
\[
\bigcap_{C\in\Lambda_x}\mathcal S_C=\Span\{I,P_x\}.
\]
\end{cor}
\begin{proof}
Apply Lemma~\ref{finite-blocks} to the nonzero members of
$\{P,I-P\}$ or $\{P_1,P_2,I-P_1-P_2\}$.
The last assertion follows from $\Lambda_x=\Gamma(P_x)$.
\end{proof}

\begin{lemma}\label{dim2t}
Let $\Lambda \subset \Conj$. Then  the intersection
$\bigcap_{C\in\Lambda}\mathcal S_C$ is two-dimensional if and only if 
there is a projection $P \in \BH$ with $P \neq I, 0$ such that 
$\bigcap_{C \in \Lambda} \SC=\Span \{P, I\}$.
\end{lemma}
\begin{proof}
Suppose that $\dim \bigl(\bigcap_{C \in \Lambda} \SC\bigl)=2$. By Lemma \ref{I}, $I \in \bigcap_{C \in \Lambda} \SC$. There is $A \in \bigcap_{C \in \Lambda} \SC$ such that $A$ is linearly independent of $I$. Since $\bigcap_{C \in \Lambda} \SC$ is closed under the involution $*$, we may assume that $A$ is self-adjoint.
Suppose that $|\sigma(A)| \ge 3$. Let $\lambda_1, \lambda_2, \lambda_3 \in \sigma(A)$ be pairwise distinct. Since $\sigma(A)$ is a compact Hausdorff space, there are mutually disjoint open neighborhoods $U_i$ of $\lambda_i$ for $i=1,2,3$. By Urysohn's Lemma, for each $i=1,2,3$, there is a continuous function $f_i : \sigma(A) \to \mathbb{R}$ such that $f_i(\lambda_i)=1$,  $f_i(z)=0$ for all $z \in \sigma(A) \setminus U_i$ and $f_i(z) \ge 0$ for all $z \in \sigma(A)$. Note that $f_1f_2=f_2f_3=f_3f_1=0$. Using the continuous functional calculus, we define a positive element $B_i=f_i(A) \in B(\cH)$ for each $i=1,2,3$. We get $B_1B_2=B_2B_3=B_3B_1=0$.  Suppose that 
\[
\alpha_1 B_1+ \alpha_2B_2+\alpha_3 B_3=0, \quad \text{for some} \ \alpha_1, \alpha_2, \alpha_3 \in \mathbb{C}.
\]
Multiplying both sides by $B_i$, we get 
\[
\alpha_iB_i^2=0.
\] 
By $\|B_i^2\|=\|B_i\|^2 \neq 0$, we obtain that $\alpha_i=0$. Hence $B_1$, $B_2$, and $B_3$ are linearly independent. On the other hand, since $A=A^*$ and $A \in \bigcap_{C \in \Lambda} \SC$, we have $CA=AC$ for every $C\in\Lambda$. Since $f_i$ is real-valued, uniform approximation by polynomials with real coefficients gives $CB_iC=B_i$. Thus $B_i \in \bigcap_{C \in \Lambda} \SC$. This leads to a contradiction with $\dim (\bigcap_{C \in \Lambda} \SC)=2 $. Hence we have $|\sigma(A)| \le 2$. Therefore there are distinct $\lambda_1,\lambda_2\in\mathbb R$ and a nontrivial projection $P$ such that 
\begin{equation*}
A=\lambda_1 P+ \lambda_2(I-P)=(\lambda_1-\lambda_2)P+ \lambda_2I.
\end{equation*}
As $A$ is linearly independent of $I$, we get $P \neq I, 0$.  For any $C \in \Lambda$, we have
\[
\lambda_1 CPC+ \lambda_2(I-CPC)=CAC=A=\lambda_1 P+ \lambda_2(I-P).
\]
It is easy to check that $CPC$ is a projection. By the uniqueness of the spectral decomposition of $A$, we have $CPC=P$ for any $C \in \Lambda$. Thus $P \in \bigcap_{C \in \Lambda} \SC$. 
Since $\dim(\bigcap_{C \in \Lambda} \SC)=2$, we have 
\[
\bigcap_{C \in \Lambda} \SC=\Span \{P, I\}.
\]
To prove the converse implication, suppose that $\bigcap_{C \in \Lambda} \SC=\Span \{P, I\}$ with $P \neq I, 0$. Then  it is clear that  $\dim (\bigcap_{C \in \Lambda} \SC)=2$.
\end{proof}

\begin{lemma}\label{scalar-preservation}
Let $T:\BH\to\BH$ be a bijective complex-linear map.
Suppose that either of the following holds:
\begin{enumerate}
\item There exists a bijection $\psi:\Conj\to\Conj$ satisfying
$T(\mathcal S_C)=\mathcal S_{\psi(C)}$ for every $C\in\Conj$.
\item There exists a bijection
$\varphi:\mathcal E(\cH)\to\mathcal E(\cH)$ satisfying
$T(\mathcal D_{\mathbf e})=\mathcal D_{\varphi(\mathbf e)}$
for every $\mathbf e\in\mathcal E(\cH)$.
\end{enumerate}
Then $T(I)=c I$ for some $c \in\bC\setminus\{0\}$.
\end{lemma}
\begin{proof}
By Lemmas~\ref{GP1} and~\ref{L1} and Corollary~\ref{comp},
\[
\bigcap_{\mathbf e\in\mathcal E(\cH)}\mathcal D_{\mathbf e}
=\bigcap_{C\in\Conj}\mathcal S_C=\bC I.
\]
Under either hypothesis,  we have $T(\bC I)=\bC I$.
Since $T$ is a bijection, there is $c \in \bC\setminus\{0\}$ such that $T(I)=c I$.
\end{proof}

\subsection{Quotient sets and lifts}

Define an equivalence relation on $\mathcal E(\cH)$ by
\[
\mathbf e\sim\mathbf f
\quad\Longleftrightarrow\quad
\mathcal D_{\mathbf e}=\mathcal D_{\mathbf f}.
\]

The equivalence class of $\mathbf e \in \mathcal{E}(\cH)$ is denoted by $[\mathbf e]$, and we write $\mathcal{E}(\cH)/ \sim$ for the corresponding quotient set. For $[\mathbf e]\in\mathcal E(\cH)/{\sim}$, set
\[
\mathcal D_{[\mathbf e]}:=\mathcal D_{\mathbf e}.
\]
This is well defined by the definition of $\sim$.  Let $\mathfrak{S}_N$ be the group of all permutations of  $\mathcal{I}_N$. We emphasize that, throughout this paper, an orthonormal basis is regarded as an ordered sequence. Thus, for an orthonormal basis $(e_n)$ and a permutation $\pi \in \mathfrak{S}_N$, the sequences $(e_n)$ and $(e_{\pi(n)})$ are regarded as distinct elements of $\mathcal{E}(\cH)$ unless $\pi$ is the identity permutation. Hence, $(e_n) \sim (f_n)$ if and only if there exist a permutation $\pi\in \mathfrak{S}_N$ and $(\alpha_n) \in \mathbb T^{\mathcal{I}_N}$ such that $f_n=\alpha_n e_{\pi(n)}$ for all $n$.  Let $q:\mathcal{E}(\cH) \to \mathcal{E}(\cH)/ \sim $ be the canonical quotient map. Fix a section $s:\mathcal{E}(\cH)/ \sim  \to \mathcal{E}(\cH)$ of $q$. 
\begin{lemma}\label{lift}
Suppose that $\mathcal{F}:\mathcal{E}(\cH)/ \sim \to \mathcal{E}(\cH)/ \sim$
is a bijection. Then there is a bijection $\varphi:\mathcal{E}(\cH) \to \mathcal{E}(\cH)$ such that $q\circ \varphi=\mathcal{F} \circ q$.
\end{lemma}
\begin{proof}
   For any $\pi_1, \pi_2 \in \mathfrak S_N$, we define $(\pi_1\pi_2)(n)=\pi_1(\pi_2(n))$. The action of $\mathfrak S_N$ on $\mathbb T^{\mathcal{I}_N}$ is given by $(\pi\cdot\beta)_n=\beta_{\pi^{-1}(n)}$.
Thus the multiplication in the semidirect product is
\[
(\alpha,\pi_1)(\beta,\pi_2)
=\bigl((\alpha_n\beta_{\pi_1^{-1}(n)})_n,\pi_1\pi_2\bigr).
\]
We define a left action of  the group $\mathbb T^{\mathcal{I}_N}\rtimes \mathfrak{S}_N$ on $\mathcal{E}(\cH)$ by
\[
((\alpha_n), \pi)   \cdot (e_n)=(\alpha_n e_{\pi^{-1}(n)}), 
\]
for $((\alpha_n), \pi) \in \mathbb T^{\mathcal{I}_N}\rtimes \mathfrak{S}_N$ and $(e_n) \in \mathcal{E}(\cH)$. Then the group $\mathbb T^{\mathcal{I}_N}\rtimes \mathfrak{S}_N$ acts freely and transitively on each equivalence class. Hence, every $(e_n) \in \mathcal{E}(\cH)$ can be written uniquely in the form
\[
(e_n)=g\cdot s(q((e_n))).
\]
Consequently, we may identify $\mathcal{E}(\cH)$ with $(\mathcal{E}(\cH)/\sim) \times (\mathbb T^{\mathcal{I}_N}\rtimes \mathfrak{S}_N)$.
Under this identification, define $\varphi:\mathcal{E}(\cH) \to  \mathcal{E}(\cH)$
by
\[
\varphi \bigl(g\cdot s(x)\bigr)=g\cdot s(\mathcal{F}(x)), \qquad x \in\mathcal{E}(\cH)/\sim, \quad  g\in \mathbb T^{\mathcal{I}_N}\rtimes \mathfrak{S}_N.
\]
Then $\varphi$ is a bijection and $q\circ \varphi=\mathcal{F} \circ q$.  
\end{proof}

\begin{lemma}\label{phase-fibres}
For $C,D\in\Conj$, we have $\mathcal S_C=\mathcal S_D$ if and only if
$D=e^{i\theta} C$ for some $e^{i\theta} \in\mathbb T$.
\end{lemma}
\begin{proof}
If $D=e^{i\theta} C$ with $|e^{i\theta}|=1$, then
$DAD=e^{i\theta}\overline{e^{i\theta}}\,CAC=CAC$ for every $A\in\BH$.
Conversely, suppose that $\mathcal S_C=\mathcal S_D$ and choose an orthonormal basis $(e_i) \in \mathcal{E}(\cH)$ such that $Ce_i=e_i$ for any $i$. Lemma \ref{projectionC} gives $e_i\otimes e_i\in\SC=\mathcal S_D$, and hence $De_i=\lambda_i e_i$ for some $\lambda_i\in\mathbb T$.
For any distinct $i,j\in\mathcal I_N$, set $u_{ij}=(e_i+e_j)/\sqrt2$.
Since $Cu_{ij}=u_{ij}$, Lemma \ref{projectionC} gives $u_{ij}\otimes u_{ij}\in\SC=\mathcal S_D$.
Thus $Du_{ij}=\mu_{ij}u_{ij}$ for some $\mu_{ij}\in\mathbb T$, and
\[
\frac{\mu_{ij}e_i+\mu_{ij}e_j}{\sqrt2}
=Du_{ij}
=\frac{\lambda_i e_i+\lambda_j e_j}{\sqrt2}.
\]
Therefore $\lambda_i=\lambda_j$ for every $i\ne j$.
Put $\lambda=\lambda_1$. For every finite linear combination $x=\sum_i\alpha_i e_i$, we have
\[
Dx=\sum_i\overline{\alpha_i}\,De_i
=\lambda\sum_i\overline{\alpha_i}e_i=\lambda Cx.
\]
By continuity, the same identity holds for every $x\in\cH$.
Thus $D=\lambda C$.
\end{proof}

Define an equivalence relation on $\Conj$ by
\[
C \sim D
\quad\Longleftrightarrow\quad
\mathcal S_C=\mathcal S_D.
\]
Let $q_c:\Conj  \to \Conj/ \sim $ be the canonical quotient map. Fix a section $s_c :\Conj / \sim  \to \Conj$ of $q_c$.

\begin{lemma}\label{conjugation-lift}
Suppose that $\mathcal G:\Conj/ \sim  \to \Conj/ \sim $ is a bijection. Then there is a bijection $\psi:\Conj\to\Conj$ 
such that 
\[
q_{c}\circ\psi=\mathcal G\circ q_{c}.
\]
\end{lemma}
\begin{proof}
By Lemma~\ref{phase-fibres}, every conjugation has a unique representation $e^{i\theta} s_{c}(q_c(C))=C$,
with $e^{i\theta} \in\mathbb T$.
The formula
\[
\psi(e^{i\theta}  s_{c}(q_c(C)))
=e^{i\theta}  s_{c}(\mathcal G(q_c(C)))
\]
defines the required bijection.
\end{proof}

\section{Proof of Theorem~\ref{main0}}\label{section3}

Throughout this section, $\dim\cH=2$, and we use the fixed identification
$\cH=\bC^2$.


We denote by
\[
U(2)=\{U\in M_2(\mathbb C)\mid U^*U=I\}
\]
the unitary group of $2\times2$ complex matrices, and by
\[
SO(3)=\{R\in M_3(\mathbb R)\mid R^TR=I,\ \det R=1\}
\]
the special orthogonal group of $\mathbb R^3$.
For a real subspace $E\subset\bR^3$, write
\[
E_{\bC}=E+iE,\qquad
E\cdot\sigma=\{x\cdot\sigma\mid x\in E\},\qquad
E_{\bC}\cdot\sigma=\{x\cdot\sigma\mid x\in E_{\bC}\}.
\]
Thus $E_{\bC}\cdot\sigma=\Span(E\cdot\sigma)$.
The rank-one projections in $M_2(\bC)$ are precisely the matrices
\[
P=\frac12(I+r\cdot\sigma),
\qquad r\in\bR^3,\quad \|r\|=1.
\]

\begin{proposition}[{\cite[Proposition 5.5]{Land}}]\label{Bloch_representation}
For each $U\in U(2)$, there exists a unique $R_U\in SO(3)$ such that
\[
U(x\cdot\sigma)U^*=(R_Ux)\cdot\sigma\qquad(x\in\bR^3).
\]
The map $U\mapsto R_U$ is a surjective group homomorphism from
$U(2)$ to $SO(3)$.
\end{proposition}

\begin{lemma}\label{plane}
We have
\[
\begin{aligned}
&\{\mathcal S_C\mid C\in\Conj\}\\
&\qquad=\{\Span\{ I, E \cdot \sigma\}\mid 
E\subset\bR^3,\ \dim_{\bR}E=2\}.
\end{aligned}
\]
\end{lemma}

\begin{proof}
Let $K$ be a conjugation associated with the standard basis. For any $C \in \Conj$, there is $U \in U(2)$ such that $C=UKU^{*}$. (Let $(f_i)$ be an orthonormal basis such that $Cf_i=f_i$. Define $U \in U(2)$ by $Ue_i=f_i$ for $i=1,2$, where $(e_i)$ is the standard basis.) Then $\SC=U\mathcal{S}_KU^{*}$. Directly from the Pauli matrices,  
\[
\mathcal S_K=\Span\{I,\sigma_1,\sigma_3\}
=\Span\{I, E_0 \cdot \sigma\},
\]
where $E_0=\{(x_1,0,x_3)\mid x_1,x_3\in\bR\}$.
By Proposition \ref{Bloch_representation}, we have
\begin{equation*}
\begin{split}
\SC&=U\mathcal{S}_KU^{*}\\
&=U\Span\{I, E_0 \cdot \sigma\}U^{*}=\Span\{I, U(E_0 \cdot \sigma)U^{*}\}\\
&=\Span\{I, R_U E_0 \cdot \sigma \}.
\end{split}
\end{equation*}
Since $R_U \in SO(3)$, $R_U E_0$ is  a two-dimensional subspace of $\bR^3$.

We prove the converse. Let $E$ be a two-dimensional subspace of $\bR^3$. Choose a unit vector $z \in \bR^3$ such that $E=\{x \in \bR^3 \mid \langle x, z\rangle=0\}$. Then there is $R \in SO(3)$ such that $R(0,1,0)=z$. As  the elements of $SO(3)$ preserve the Euclidean inner product on $\bR^3$, we have  
\[
RE_0=R(\{x \in \bR^3 \mid \langle x, (0,1,0)\rangle =0\})=\{x \in \bR^3 \mid \langle x, z\rangle=0\}=E.
\]
By Proposition \ref{Bloch_representation}, there is $U \in U(2)$ such that $R=R_U$. Define $D=UKU^{*}$. Then $D \in \Conj$ and 
\[
\mathcal{S}_D=U\mathcal{S}_KU^{*}=\Span\{I, R_U E_0 \cdot \sigma \}=\Span\{I, E \cdot \sigma \}
\]
by the same argument as above. 
\end{proof}

\begin{proof}[{\textbf {Proof of Theorem \ref{main0}}}] 
\textbf{(1) $\Rightarrow$ (2).}
Let $(e_1,e_2) \in \mathcal{E}(\cH)$. Corollary \ref{projection-intersections} shows that 
\[
\bigcap_{C \in \Lambda_{e_1}} \SC=\Span\{e_1 \otimes e_1, I \}=\mathcal{D}_{(e_1,e_2)}. 
\]
Since $T\bigl(\bigcap_{C \in \Lambda_{e_1}} \SC\bigr)=\bigcap_{C \in \Lambda_{e_1}} \mathcal{S}_{\psi(C)}$ and $\dim \bigcap_{C \in \Lambda_{e_1}} \SC=2$, we have $\dim \bigcap_{C \in \Lambda_{e_1}} \mathcal{S}_{\psi(C)}=2$. By Lemma \ref{dim2t},  there is $e'_1 \in \cH$ with $\|e'_1\|=1$ such that 
\[
\bigcap_{C \in \Lambda_{e_1}} \mathcal{S}_{\psi(C)}=\Span\{e'_1 \otimes e'_1, I \}=\mathcal{D}_{(e'_1,e'_2)},
\]
where $(e'_1, e'_2) \in \mathcal{E}(\cH)$. 
Hence, we have 
\begin{equation}\label{diago}
T(\mathcal{D}_{(e_1,e_2)})=\mathcal{D}_{(e'_1,e'_2)}, 
\end{equation}
for any $(e_1,e_2)  \in \mathcal{E}(\cH)$. Since a similar argument holds for $T^{-1}$, there is a bijective map $\mathcal{F}: \mathcal{E}(\cH)/ \sim \to \mathcal{E}(\cH)/ \sim$ such that 
\[
T(\mathcal{D}_{[e_n]})=\mathcal{D}_{\mathcal{F}([e_n])}.
\]
Applying Lemma \ref{lift}, there is a bijection $\varphi$ such that $q\circ \varphi=\mathcal{F} \circ q$. Therefore we have
\[
T(\mathcal{D}_{(e_n)})=T(\mathcal{D}_{[e_n]})=\mathcal{D}_{\mathcal{F}([e_n])}=\mathcal{D}_{[\varphi((e_n))]}=\mathcal{D}_{\varphi((e_n))}.
\] 
This completes the proof of  $(1)\Rightarrow(2)$. \\

\medskip
\noindent \textbf{(2) $\Rightarrow$ (3).}
By Lemma \ref{scalar-preservation}, there is $c \in \mathbb{C} \setminus \{0\}$ such that $T(I)=cI$. Since $T$ is a complex linear map, there are linear maps $\ell: \mathbb{C}^3 \to \mathbb{C}$ and $L: \mathbb{C}^3 \to \mathbb{C}^3$ such that 
\[
T(x \cdot \sigma)=\ell(x)I+L(x)\cdot \sigma.
\]
Since $T$ is a bijection, $L$ is an invertible matrix. We shall show that for any $y \in \mathbb{R}^3 \setminus \{0\}$, there are $z \in \mathbb{R}^3$ with $\|z\|=1$ and $\lambda \in \mathbb{C}\setminus \{0\}$ such that 
\begin{equation*}
    L(y)=\lambda z.
\end{equation*}
Indeed, for any $y=(y_1, y_2, y_3) \in \mathbb{R}^3 \setminus \{0\}$, there is $\mathbf e=(e_1, e_2) \in \mathcal{E}(\cH)$ such that $\frac{1}{2}(I+\frac{y}{\|y\|}\cdot \sigma)=e_1\otimes e_1$. Then 
\[
T(\Span\{e_1\otimes e_1, I\})=\Span\{f_1 \otimes f_1, I\},
\]
where $\varphi(\mathbf e)=(f_1, f_2)$. Thus there is $z=(z_1, z_2, z_3) \in \mathbb{R}^3$ with $z_1^2+z_2^2+z_3^2=1$ such that $f_1 \otimes f_1=\frac{1}{2}(I+z \cdot \sigma)$. 
This implies 
\begin{multline*}
\Span\{I, L(y)\cdot \sigma\}=T(\Span\{I, y/\|y\| \cdot \sigma\})\\=T(\Span\{e_1\otimes e_1, I\})=\Span\{f_1 \otimes f_1, I\}=\Span\{I, z \cdot \sigma\}.
\end{multline*}
Since $\{I, \sigma_1, \sigma_2, \sigma_3\}$ are linearly independent,  we have $L(y)=\lambda z$ for some $\lambda \in \mathbb{C} \setminus \{0\}$. 

Let $u_1=(1,0,0)$ , $u_2=(0,1,0)$ and $u_3=(0,0,1)$ be the standard basis of $\mathbb{R}^3$. By above,  there are $b_i \in \mathbb{R}^3$ with $\|b_i\|=1$ such that 
\[
L(u_i)=\lambda_i b_i
\]
for some $\lambda_i \in \mathbb{C} \setminus \{0\}$ for $i=1,2,3$. 
Since $u_1+u_2+u_3 \in \mathbb{R}^3 \setminus \{0\}$, we also get that there is $b \in \mathbb{R}^3$ such that 
\[
L(u_1+u_2+u_3)=\lambda b
\]
for some $\lambda \in \mathbb{C}\setminus\{0\}$. Since $L$  is an invertible matrix and $\{u_1, u_2, u_3\}$  is an orthonormal basis, the vectors $\{b_1, b_2, b_3\}$ are linearly independent. As $b_i \in \mathbb{R}^3$ and $b \in \mathbb{R}^3$, there are $r_i \in \mathbb{R}$ such that $b=r_1b_1+r_2b_2+r_3b_3$. This implies that
\[
\lambda_1 b_1+\lambda_2 b_2+\lambda_3 b_3=L(u_1+u_2+u_3)=\lambda b= \lambda (r_1b_1+r_2b_2+r_3b_3).
\]
We get $\lambda_i=\lambda r_i$ for any $i=1,2,3$. Thus we have $r_i \neq 0$. We define $B \in GL(3,\mathbb{R})$ by $Bu_i=r_ib_i$ for any $i=1,2,3$. Then we have $L=\lambda B$. Therefore we have
\[
T(aI+x\cdot \sigma)=c aI+\ell(x)I+\lambda Bx \cdot \sigma.
\] 

\medskip
\noindent\textbf{(3) $\Rightarrow$ (1).}
Applying Lemma \ref{plane}, for any $C \in \Conj$, we have 
\[
\SC=\Span\{I, E\cdot\sigma\},  
\]
where $E$ is a  two-dimensional subspace of $\mathbb{R}^3$. For any $a\in\bC$ and $x \in E_{\bC}$, we have
\[
T(aI+x \cdot \sigma)=c aI+\ell(x)I+\lambda Bx \cdot \sigma \in \Span\{ I, BE \cdot \sigma\}.
\]
Since $B \in GL(3, \mathbb{R})$, $BE$ is a  two-dimensional subspace of $\mathbb{R}^3$. Hence there is $D \in \Conj$ such that $\Span\{ I, BE \cdot \sigma\}=\mathcal{S}_D$ by Lemma \ref{plane}. As $T$ is injective and both $\SC$ and $\mathcal S_D$ have complex dimension three, we obtain $T(\SC)=\mathcal{S}_D$. By applying a similar argument to $T^{-1}$,  this induces a bijection
$\mathcal G: \Conj / \sim \to \Conj / \sim$ characterized by
$\mathcal G([C])=[D]$ whenever $T(\mathcal S_C)=\mathcal S_D$ by Lemma~\ref{phase-fibres}. 
Lemma~\ref{conjugation-lift} gives a bijective map $\psi$ with
$T(\mathcal S_C)=\mathcal S_{\psi(C)}$ for every $C\in\Conj$.
\end{proof}

\section{Proof of Theorem~\ref{main}}\label{section4}

Throughout this section, $\dim\cH\ge3$, and $T:\BH\to\BH$ is a
bounded bijective complex-linear map.
By the bounded inverse theorem, $T^{-1}$ is also bounded.
When $T$ satisfies condition (1) or (2), its inverse satisfies the same
condition with $\psi^{-1}$ or $\varphi^{-1}$, respectively.


\begin{proof}[Proof of Theorem~\ref{main}: $(3)\Rightarrow(1)$]
Suppose that (3) holds, and define
\[
J(A)=
\begin{cases}
UAU^*,&\text{if }U\text{ is unitary},\\
UA^*U^*,&\text{if }U\text{ is antiunitary},
\end{cases}
\qquad A\in\BH.
\]
Thus $T(A)=cJ(A)+f(A)I$, and $J$ is a complex-linear bijection.
Set $\psi(C)=UCU^*$ for $C\in\Conj$.
The map $\psi$ is a bijection, and the identity
\[
\psi(C_1)\psi(C_2)=UC_1C_2U^*
\]
shows that it preserves commutativity in both directions.
For either choice of $U$,
\[
\psi(C)(UAU^*)\psi(C)=U(CAC)U^*,\qquad
(UAU^*)^*=UA^*U^*.
\]
Since $\mathcal S_C$ is closed under taking adjoints,
\[
A\in\mathcal S_C
\quad\Longleftrightarrow\quad J(A)\in\mathcal S_{\psi(C)}.
\]
Moreover, $c\ne0$ and $\bC I\subset\mathcal S_{\psi(C)}$ imply
\[
T(A)\in\mathcal S_{\psi(C)}
\quad\Longleftrightarrow\quad J(A)\in\mathcal S_{\psi(C)}.
\]
Thus $T^{-1}(\mathcal S_{\psi(C)})=\mathcal S_C$.
The surjectivity of $T$ gives
$T(\mathcal S_C)=\mathcal S_{\psi(C)}$ for every $C\in\Conj$. 
\end{proof}
It remains to prove $(1)\Rightarrow(2)$ and $(2)\Rightarrow(3)$ in this paper.

\subsection{The implication (1) $\Rightarrow$ (2)}

For a real Hilbert space $V$, let $\mathcal B_{\bR}(V)$ denote the
bounded real-linear operators on $V$, and set
\[
\mathcal P_{\bR}(V)=\{p\in\mathcal B_{\bR}(V)\mid p^2=p=p^*\}.
\]
The adjoint here is taken with respect to the real inner product.
We call the elements of $\mathcal P_{\bR}(V)$ real orthogonal projections
and write $\mathcal P_{\bR,1}(V)$ for those of real rank one.
Lowercase letters $p,q,r$ denote real projections, and uppercase
letters $P,Q,R$ denote complex-linear orthogonal projections on $\cH$.

For $C_0\in\Conj$, its fixed-point space
\[
V_{C_0}=\{x\in\cH\mid C_0x=x\}
\]
is a real Hilbert space with the restricted inner product.
We call $V_{C_0}$ the real form associated with $C_0$.
Every $z\in\cH$ has a unique representation $z=x+iy$ with
$x,y\in V_{C_0}$, namely
\[
x=\frac{z+C_0z}{2},\qquad y=\frac{z-C_0z}{2i}.
\]
Thus $\cH=V_{C_0}\oplus iV_{C_0}$ as real vector spaces, and
$\dim_{\bR}V_{C_0}=\dim_{\bC}\cH$.
For $p\in\mathcal P_{\bR}(V_{C_0})$, define its complex-linear extension by
\[
\widehat p(x+iy)=px+i\,py\qquad(x,y\in V_{C_0}).
\]
Then $\widehat p\in\mathcal P(\cH)$ and $\widehat pC_0=C_0\widehat p$.
Conversely, every $P\in\mathcal P(\cH)$ commuting with $C_0$ is the
extension of $P|_{V_{C_0}}$.
In particular,
\[
\rank_{\bC}\widehat p=\rank_{\bR}p,\qquad
\widehat{I_{V_{C_0}}-p}=I_{\cH}-\widehat p.
\]
We distinguish the identities $I_V$ and $I_{\cH}$ when real and
complex spaces occur in the same argument.
For a real form $V=V_{C_0}$ and $\mathbf e\in\mathcal E(\cH)$, set
\[
\mathcal P^{\bR}_{\mathbf e}(V)
=\{p\in\mathcal P_{\bR}(V)\mid\widehat p\in\mathcal P_{\mathbf e}\}.
\]

For $\mathcal A\subset\mathcal P_{\bR}(V)$, define its commutant
within the real projection set by
\[
\mathcal A'
=\{q\in\mathcal P_{\bR}(V)\mid qp=pq\text{ for every }p\in\mathcal A\},
\qquad \mathcal A''=(\mathcal A')'.
\]
These commutants consist only of projections. 

\begin{lemma}\label{doublecommutant}
Let $V$ be a real Hilbert space with $\dim_{\bR}V\ge3$.
Then the following assertions hold:
\begin{enumerate}
\item $\{0\}''=\{0,I_V\}$.
\item For every $p\in\mathcal P_{\bR}(V)\setminus\{0,I_V\}$,
\[
\{p\}''=\{0,I_V,p,I_V-p\}.
\]
\item For every $p\in\mathcal P_{\bR}(V)\setminus\{0,I_V\}$,
\[
\begin{split}
&\rank_{\bR}p=1\quad\text{or}\quad\rank_{\bR}(I_V-p)=1\\
&\qquad\Longleftrightarrow\quad
|\{p,q\}''|\le8\quad\text{for every }q\in\{p\}'.
\end{split}
\]
\end{enumerate}
\end{lemma}
\begin{proof}
(1) Let $r \in\{0\}''$. Since $\{0\}'=\mathcal P_{\bR}(V)$,
$r$ commutes with every orthogonal projection on $V$.
Suppose that $r \notin\{0,I_V\}$. Choose nonzero vectors $u\in\ran r$ and $v\in\ker r$. The projection onto $\mathbb{R}(u+v)$ does not commute with $r$, which is a contradiction. 
Hence $\{0\}''=\{0,I_V\}$. This argument proves the same assertion on every nonzero real Hilbert space, including spaces of dimension one or two.

(2) Let $r\in\{p\}''$.
Since $p\in\{p\}'$, the projection $r$ commutes with $p$ and is block
diagonal with respect to $V=\ran p\oplus\ker p$.
The projections in $\{p\}'$ are precisely the block-diagonal projections.
Each block of $r$ therefore commutes with every orthogonal projection
on the corresponding subspace. By (1), each block is zero or the identity.
This gives the four stated possibilities, all of which belong to $\{p\}''$.

(3) Let $q \in\{p\}'$. Since $p$ and $q$ commute, $V$ splits into
the four joint eigenspaces
\[
V_{ij}=\{x\in V\mid  px=ix,\ qx=jx\},
\qquad i,j\in\{0,1\}.
\]
A projection belongs to $\{p,q\}''$ exactly when it is either zero
or the identity on each nonzero $V_{ij}$. Therefore, if $k$ of
the four spaces $V_{ij}$ are nonzero, then
\[
|\{p, q\}''|=2^k.
\]

If $\rank_{\bR} p=1$, only one of $V_{11}$ and $V_{10}$
can be nonzero. Thus $k\leq 3$, and hence
\[
|\{p,q\}''|\leq 8.
\]
The same argument applies if $\rank_{\bR} (I_V-p)=1$.

Conversely, suppose that $\rank_{\bR} p \ge 2$ and $\rank_{\bR} (I_V-p)\ge 2$.
Choose rank-one projections $r \le p$ and $s \le I_V-p$, and put
$q=r+s$. Then $pq=qp$, and all four joint eigenspaces are
nonzero. Therefore,
\[
|\{p,q\}''|=2^4=16.
\]
\end{proof}

\begin{lemma}\label{commutative}
Let $\psi:\Conj\to\Conj$ be a bijection preserving commutativity
in both directions. For every $\mathbf e\in\mathcal E(\cH)$, there
exists $\mathbf f\in\mathcal E(\cH)$ such that
\[
\psi(\Lambda_{\mathbf e})=\Lambda_{\mathbf f}.
\]
\end{lemma}
\begin{proof}
Let $\mathbf e=(e_i)$ and $C_0=C_{\mathbf e}$, so that $C_0e_i=e_i$
for every $i$. Put
\[
D_0=\psi(C_0),\qquad V=V_{C_0},\qquad W=V_{D_0}.
\]
For $p\in\mathcal P_{\bR}(V)$, define
\[
C_p=(I_{\cH}-2\widehat p)C_0.
\]
Since $\widehat p$ commutes with $C_0$, the operator $C_p$ is a
conjugation commuting with $C_0$.
Conversely, if a conjugation $C$ commutes with $C_0$, then $CC_0$
is a self-adjoint unitary commuting with $C_0$. Hence
\[
\frac{I_{\cH}-CC_0}{2}
\]
is a complex-linear orthogonal projection commuting with $C_0$.
Its restriction to $V$ is a real orthogonal projection $p$ satisfying
$C=C_p$. Thus $p\mapsto C_p$ is a bijection onto the conjugations
commuting with $C_0$. Moreover,
\[
C_pC_q=C_qC_p\quad\Longleftrightarrow\quad pq=qp.
\]
The same construction applies to $D_0$ and $W$.
Therefore
\[
F(p)=\left.\frac{I_{\cH}-\psi(C_p)D_0}{2}\right|_W
\]
defines a bijection
$F:\mathcal P_{\bR}(V)\to\mathcal P_{\bR}(W)$ preserving commutativity
in both directions. It satisfies $F(0)=0$ and
\begin{equation}\label{psiF}
\psi(C_p)=\bigl(I_{\cH}-2\widehat{F(p)}\bigr)D_0.
\end{equation}

For every $\mathcal A\subset\mathcal P_{\bR}(V)$,
$F(\mathcal A'')=F(\mathcal A)''$, where the commutants are taken in $\mathcal P_{\mathbb R}(V)$ and $\mathcal P_{\mathbb R}(W)$, respectively.
Lemma~\ref{doublecommutant} (1) and (2) implies
\[
F(I_V)=I_W,\qquad F(I_V-p)=I_W-F(p).
\]
By Lemma~\ref{doublecommutant} (3), $F$ maps the projections of real
rank one or real corank one onto the projections of the same kind.
For $p\in\mathcal P_{\bR,1}(V)$, let $\widetilde F(p)$ be the unique
rank-one member of $\{F(p),I_W-F(p)\}$.
Uniqueness follows from $\dim_{\bR}W\ge3$.
The analogous construction for $F^{-1}$ is the inverse of
$\widetilde F$, so $\widetilde F$ is a bijection.
A projection and its complement have the same commutant.
Thus $\widetilde F$ preserves commutativity in both directions.
Note that for $p,q \in \mathcal P_{\bR,1}(V)$ with $p \neq q$, we have $pq=0 \Leftrightarrow pq=qp$. This implies that $\widetilde{F}$ also
preserves orthogonality in both directions.

Let $p_i$ be the real orthogonal projection onto $\bR e_i$ in $V$,
and put $r_i=\widetilde F(p_i)$.
The family $(p_i)$ is a maximal orthogonal family of rank-one projections
on $V$. Since $\widetilde F$ is bijective and preserves orthogonality
in both directions, $(r_i)$ is also maximal.
Choose a unit vector $f_i\in r_i(W)$ for each $i$.
Then $(f_i)$ is an orthonormal basis of $W$ over $\bR$, and hence
$\mathbf f=(f_i)$ is an orthonormal basis of $\cH$ over $\bC$.
Moreover, $D_0f_i=f_i$.

For $p\in\mathcal P_{\bR}(V)$,
\[
\begin{aligned}
\widehat p\in\mathcal P_{\mathbf e}
&\Longleftrightarrow pp_i=p_ip &&\text{for every }i\\
&\Longleftrightarrow F(p)F(p_i)=F(p_i)F(p) &&\text{for every }i\\
&\Longleftrightarrow F(p)r_i=r_iF(p) &&\text{for every }i\\
&\Longleftrightarrow \widehat{F(p)}\in\mathcal P_{\mathbf f}.
\end{aligned}
\]
Since $F$ is onto, we have $F\bigl(\mathcal P^{\bR}_{\mathbf e}(V)\bigr)
=\mathcal P^{\bR}_{\mathbf f}(W)$. 
Finally, Lemma~\ref{proj1} and \eqref{psiF} give
\[
\begin{aligned}
\psi(\Lambda_{\mathbf e})
&=\{\psi(C_p)\mid p\in\mathcal P^{\bR}_{\mathbf e}(V)\}\\
&=\{(I_{\cH}-2\widehat{F(p)})D_0\mid 
                         p\in\mathcal P^{\bR}_{\mathbf e}(V)\}\\
&=\{(I_{\cH}-2\widehat r)D_0\mid 
                         r\in\mathcal P^{\bR}_{\mathbf f}(W)\}\\
&=\Lambda_{\mathbf f}.
\end{aligned}
\]
\end{proof}

\begin{proof}[Proof of Theorem~\ref{main}: $(1)\Rightarrow(2)$]
Let $\mathbf e\in\mathcal E(\cH)$.
By Lemma~\ref{commutative}, there exists $\mathbf f\in\mathcal E(\cH)$
such that $\psi(\Lambda_{\mathbf e})=\Lambda_{\mathbf f}$.
Lemma~\ref{L1} and the injectivity of $T$ imply
\[
\begin{aligned}
T(\mathcal D_{\mathbf e})
&=T\Bigl(\bigcap_{C\in\Lambda_{\mathbf e}}\mathcal S_C\Bigr)
=\bigcap_{C\in\Lambda_{\mathbf e}}T(\mathcal S_C)\\
&=\bigcap_{C\in\Lambda_{\mathbf e}}\mathcal S_{\psi(C)}
=\bigcap_{D\in\Lambda_{\mathbf f}}\mathcal S_D
=\mathcal D_{\mathbf f}.
\end{aligned}
\]

Since $T$ is a bijection and a similar argument as above holds for $T^{-1}$, there is a bijection $\mathcal{F}:\mathcal{E}(\cH)/ \sim  \to \mathcal{E}(\cH)/ \sim $ defined by 
\[
T(\mathcal{D}_{[\mathbf e]})=\mathcal{D}_{\mathcal{F}([\mathbf e])}.
\]
Applying Lemma \ref{lift}, there is  a bijection $\varphi$ satisfying $q\circ \varphi=\mathcal{F} \circ q$. Therefore we have
\[
T(\mathcal{D}_{\mathbf e})=T(\mathcal{D}_{[\mathbf e]})=\mathcal{D}_{\mathcal{F}([\mathbf e])}=\mathcal{D}_{[\varphi(\mathbf e)]}=\mathcal{D}_{\varphi(\mathbf e)}.
\] 
\end{proof}

\subsection{The implication (2) $\Rightarrow$ (3)}

Throughout this subsection, assume that $\varphi:\mathcal E(\cH)\to
\mathcal E(\cH)$ is a bijection satisfying
$T(\mathcal D_{\mathbf e})=\mathcal D_{\varphi(\mathbf e)}$
for every $\mathbf e\in\mathcal E(\cH)$.

The implication (2) $\Rightarrow$ (3) also follows from the characterization, due to Bre\v{s}ar and \v{S}emrl \cite[Theorem~2]{Brs}, of linear maps that preserve normal operators. Indeed, normal operators with finite spectrum are norm dense in the set of all normal operators, and every normal operator with finite spectrum belongs to some diagonal algebra $\mathcal D_{\mathbf e}$. Hence, condition (2), together with the boundedness of $T$, implies that $T$ maps normal operators to normal operators.

Nevertheless, we include a direct proof using rank-one projections. Under condition (2), we construct a bijection on the set of rank-one projections that preserves orthogonality in both directions. Uhlhorn's theorem then yields the unitary or antiunitary operator appearing in the description of $T$.

For a projection $P\in\mathcal P(\cH)$, its complex operator commutant is
\[
\mathcal C(P)=\{A\in\BH\mid AP=PA\}.
\]
This is a complex-linear subspace of $\BH$; it differs from the
commutant within the real projection set used in the preceding subsection.

\begin{lemma}\label{projection1}
Let $P_1, P_2 \in \BH $ be projections such that $P_1(\cH) \perp P_2(\cH)$. 
Then there are projections $Q_1, Q_2 \in \BH$
such that 
\[
T(\Span \{P_1, I\} )=\Span \{Q_1,I\},
\]
 and 
\[
T(\Span \{P_2, I\} )=\Span \{Q_2,I\}.
\]
For any such choice, at least one of the following holds:
\[
Q_1Q_2=0, \ Q_1Q_2=Q_1, \ Q_1Q_2=Q_2, \ \text{or} \ (I-Q_1)(I-Q_2)=0.
\]
\end{lemma}

\begin{proof}
Lemma~\ref{scalar-preservation} gives $T(\bC I)=\bC I$.
For $i=1,2$, let
\[
\mathcal E_i=\{\mathbf e\in\mathcal E(\cH)\mid P_i\in\mathcal D_{\mathbf e}\}.
\]
Since $P_1$ and $P_2$ are orthogonal, $\mathcal E_1\cap\mathcal E_2$ is nonempty.
Then we shall show that 
\begin{equation}\label{cap1}
\bigcup_{\mathbf e\in\mathcal E_i}\Lambda_{\mathbf e}=\Gamma(P_i).
\end{equation}
Let $C \in \Lambda_{(e_n)}$ with $(e_n) \in \mathcal{E}_i$. For any $x \in P_i(\cH)$, there are complex coefficients $\alpha_n$ such that $x=\sum_{\{n \mid e_n \in P_i(\cH)\}} \alpha_ne_n$. We have 
\[
Cx=C\Bigl(\sum_{\{n \mid e_n \in P_i(\cH)\}} \alpha_ne_n\Bigr)=\sum_{\{n \mid e_n \in P_i(\cH)\}}\overline{\alpha_n} \rho(n) e_n \in P_i(\cH).
\]
 Thus $CP_i(\cH)=P_i(\cH)$. By Lemma~\ref{projectionC} and the definition $\Gamma(P_i)$, we have $C \in \Gamma(P_i)$. 
To prove the opposite inclusion, let $C \in \Conj$ with $CP_i(\cH)=P_i(\cH)$. The restrictions of $C$ to $P_i(\cH)$ and $\ker P_i$ are conjugations. Combining fixed orthonormal bases of these subspaces gives $(e_n) \in \mathcal{E}_i$ such that $Ce_n=e_n$. Then $C \in \Lambda_{(e_n)}$. 

By a similar argument, using fixed orthonormal bases of the nonzero subspaces $P_1(\cH)$, $P_2(\cH)$, and $(I-P_1-P_2)(\cH)$, we also obtain
\begin{equation}\label{cap12}
\bigcup_{\mathbf e\in\mathcal{E}_1 \cap \mathcal{E}_2}\Lambda_{\mathbf e}
=\Gamma(P_1,P_2).
\end{equation}

By Lemma \ref{L1} and Corollary \ref{projection-intersections}, and \eqref{cap1}, we have 
\begin{equation*}
    \begin{split}
        \Span\{P_i, I\}&=\bigcap_{C \in \Gamma(P_i) } \SC\\&
        =\bigcap_{\mathbf e \in \mathcal{E}_i} \bigcap_{C \in \Lambda_{\mathbf e}} \SC
        =\bigcap_{\mathbf e\in \mathcal{E}_i} \mathcal{D}_{\mathbf e}. 
    \end{split}
\end{equation*}
Thus we get
\begin{equation*}
    \begin{split}
        T(\Span\{P_i, I\})
        &=T\Bigl(\bigcap_{\mathbf e \in \mathcal{E}_i} \mathcal{D}_{\mathbf e}\Bigl)\\&=\bigcap_{\mathbf e \in \mathcal{E}_i} \mathcal{D}_{\varphi(\mathbf e)}=\bigcap_{\mathbf e \in \mathcal{E}_i} \bigcap_{C \in \Lambda_{\varphi(\mathbf e)}}\SC. 
    \end{split}
\end{equation*}
Since $\dim \bigl(\bigcap_{\mathbf e \in \mathcal{E}_i} \bigcap_{C \in \Lambda_{\varphi(\mathbf e)}}\SC\bigr) \le 2$,  Lemma \ref{dim1} and Lemma \ref{dim2t} show that there are projections $Q_i \in \BH$ such that 
\[
T(\Span\{P_i, I\})=\Span\{Q_i, I\}, \quad i=1,2.
\]
Note that
\begin{equation}\label{Q1Q2}
    Q_i \in \bigcap_{\mathbf e \in \mathcal{E}_i} \bigcap_{C \in \Lambda_{\varphi(\mathbf e)}}\SC \subset \bigcap_{\mathbf e \in \mathcal{E}_1 \cap \mathcal{E}_2}\bigcap_{C \in \Lambda_{\varphi(\mathbf e)}}\SC.
\end{equation}
 We have
\begin{equation*}
    \begin{split}
\bigcap_{\mathbf e \in \mathcal{E}_1 \cap \mathcal{E}_2}\bigcap_{C \in \Lambda_{\varphi(\mathbf e)}}\SC&=\bigcap_{\mathbf e \in \mathcal{E}_1 \cap \mathcal{E}_2}\mathcal{D}_{\varphi(\mathbf e)}=\bigcap_{\mathbf e \in \mathcal{E}_1 \cap \mathcal{E}_2}T(\mathcal{D}_{\mathbf e})\\&=T\Bigl(\bigcap_{\mathbf e \in \mathcal{E}_1 \cap \mathcal{E}_2}\mathcal{D}_{\mathbf e}\Bigr)=T\Bigl(\bigcap_{\mathbf e \in \mathcal{E}_1 \cap \mathcal{E}_2}\bigcap_{C \in \Lambda_{\mathbf e}}\SC\Bigr)\\&=T\Bigl(\bigcap_{C \in \Gamma(P_1,P_2)} \SC\Bigr)=T(\Span\{ P_1, P_2, I\})=\Span\{ Q_1, Q_2, I\}
        \end{split}
\end{equation*}
by \eqref{cap12} and Corollary \ref{projection-intersections}. 
Let $\mathbf e \in \mathcal{E}_1 \cap \mathcal{E}_2$. Then we get $\Span\{ P_1, P_2, I\} \subset \mathcal{D}_{\mathbf e}$. Therefore 
\[
\Span\{ Q_1, Q_2, I\} =T(\Span\{ P_1, P_2, I\}) \subset T(\mathcal{D}_{\mathbf e})=\mathcal{D}_{\varphi(\mathbf e)}.
\]
This implies that $Q_1Q_2=Q_2Q_1$. 
Fix $C \in \Lambda_{\varphi(\mathbf e)}$ where $\mathbf e \in \mathcal{E}_1 \cap \mathcal{E}_2$. By \eqref{Q1Q2}, we get $CQ_iC=Q_i$ for $i=1,2$. Hence we get
\[
CQ_1Q_2C=CQ_1CCQ_2C=Q_1Q_2=Q_2Q_1=(Q_1Q_2)^*.
\]
It follows that 
\[
Q_1Q_2 \in \bigcap_{\mathbf e \in \mathcal{E}_1 \cap \mathcal{E}_2}\bigcap_{C \in \Lambda_{\varphi(\mathbf e)}}\SC=\Span\{ Q_1, Q_2, I\}.
\]
Suppose that $Q_1Q_2 \neq 0$. There are $\alpha, \beta, \gamma \in \mathbb{C}$ such that
\[
Q_1Q_2=\alpha Q_1 +\beta Q_2+\gamma I.
\]
Multiplying from the left by $Q_1$, we obtain
\[
(1-\beta)Q_1Q_2=(\alpha+\gamma)Q_1.
\]
Multiplying from the right by $Q_2$, we get $(1-\beta)Q_1Q_2=(\alpha+\gamma)Q_1Q_2$. Since $Q_1Q_2 \neq 0$, $1-\beta=\alpha+\gamma$. We consider two cases: (i)$\beta \neq 1$ and (ii) $\beta =1$.\\
\begin{itemize}
\item[(i)]When $\beta \neq 1$, we get 
\[
Q_1Q_2=Q_1.
\]

\item[(ii)] When $\beta =1$, we have $\gamma=-\alpha$. Thus 
\[
Q_1Q_2=\alpha Q_1 + Q_2-\alpha I.
\]
\begin{itemize}
\item[(ii.1)] When $\alpha \neq 1$, multiplying from the right by $Q_2$, we obtain $Q_1Q_2=\alpha Q_1Q_2 + Q_2-\alpha Q_2 $. It follows that
\[
(1-\alpha) Q_1Q_2=(1-\alpha) Q_2. 
\]
Since $\alpha \neq 1$, we get
\[
Q_1Q_2=Q_2.
\]
\item[(ii.2)]When $\alpha =1$, we have $Q_1Q_2=Q_1 + Q_2- I $, which implies that
\[
(I-Q_1)(I-Q_2)=0.
\]
\end{itemize}
\end{itemize}
\end{proof}
\begin{lemma}\label{comm}
If $P_1,P_2\in\PH$ satisfy $P_1P_2=0$, then there exist
$Q_1,Q_2\in\PH$ with $Q_1Q_2=0$ such that
\[
T(P_i)\in\Span\{I,Q_i\}\qquad(i=1,2).
\]
\end{lemma}
\begin{proof}
Since $P_1P_2=0$, by Lemma \ref{projection1}, there are projections $Q_1, Q_2 \in \BH \setminus \{I,0\}$ such that $T(\Span \{P_i, I\} )=\Span \{Q_i,I\}$ for $i=1,2$. Replacing one or both projections by their complements if necessary, we may assume that $Q_1Q_2=0$.
Since $T(\bC I)=\bC I$ and $T$ is injective, $T(P_i)$ is non-scalar. Thus
\[
T(P_i)=a_iQ_i+b_iI,\qquad a_i\ne0.
\]
Fix $i\in\{1,2\}$. We shall show that $T(\mathcal C(P_i))\subset\mathcal C(Q_i)$.
Let $B\in\mathcal C(P_i)$. We assume that $B$ is a self-adjoint operator.
Then there is a spectral measure $E_B$ on $\sigma(B)$ such that
\[
B=\int_{\sigma(B)}\lambda\,dE_B(\lambda).
\]
Since $BP_i=P_iB$,  for any Borel set $\Delta \subset \sigma(B)$,  $P_i$ commutes with $E_B(\Delta)$. Hence $P_i E_B(\Delta)$ is a projection for any $\Delta$.   Fix $\Delta \subset \sigma(B)$. Assume that $P_i E_B(\Delta) \neq 0$. As $P_i$ is  of rank $1$, we have $P_i(\cH)  \subset E_B(\Delta)(\cH)$. In this case, $P_i (\cH) \perp (I-E_B(\Delta))(\cH)$. Otherwise we have $P_i E_B(\Delta) =0$, which means $P_i(\cH) \perp E_B(\Delta)(\cH)$. In either case,  by Lemma \ref{projection1}, there is a projection $Q(\Delta) \in \BH$
such that  
\[
T(\Span\{E_B(\Delta) , I\})=\Span \{ Q(\Delta), I\},
\]
which satisfies $Q_iQ(\Delta)=0$ or $Q_iQ(\Delta)=Q_i$ or $Q_iQ(\Delta)=Q(\Delta)$ or $(I-Q_i)(I-Q(\Delta))=0$. Therefore we get
\[
T(E_B(\Delta))\in\mathcal C(Q_i)
\]
for every Borel set $\Delta\subset\sigma(B)$.
We define the inclusion map $\operatorname{id}: \sigma(B) \to \bC$ by $\operatorname{id}(\lambda)=\lambda$ for any $\lambda \in \sigma(B)$. 
Let $f_n=\sum_{j=1}^{k_n}a_{n,j}\chi_{\Delta_{n,j}}$
be real-valued Borel simple functions such that
$\|f_n-\operatorname{id}\|_\infty\to0$.
Put $B_n=\sum_{j=1}^{k_n}a_{n,j}E_B(\Delta_{n,j})$.
Then we have
\[
T(B_n)=T\Bigl(\sum_{j=1}^{k_n}a_{n,j}E_B(\Delta_{n,j})\Bigr)
=\sum_{j=1}^{k_n}a_{n,j}T(E_B(\Delta_{n,j})).
\]
This implies that 
\[
T(B_n) \in \mathcal{C}(Q_i)
\]
for all $n$. Since $B_n \to B$ in $\BH$ and $T$ is a continuous map, we get $T(B_n) \to T(B)$. Since $\mathcal{C}(Q_i)$ is a closed subspace of $\BH$, we get $T(B) \in \mathcal{C}(Q_i)$. Hence we get 
\[
T(\{ B \in \mathcal{C}(P_i) \mid B=B^{*}\} ) \subset \mathcal{C}(Q_i).
\]
Let $B \in \mathcal{C}(P_i)$. Note that $B_1:=\frac{B+B^{*}}{2} \in \mathcal{C}(P_i)$ and $B_2:=\frac{B-B^{*}}{2i} \in \mathcal{C}(P_i)$ since $P_i=P_i^{*}$. Since $T(B)Q_i=T(B_1+iB_2)Q_i=T(B_1)Q_i+iT(B_2)Q_i=Q_iT(B_1)+iQ_iT(B_2)=Q_iT(B)$, we get
\[
T( \mathcal{C}(P_i)  ) \subset \mathcal{C}(Q_i).
\]
By \cite[Lemma 3.2]{Oml},  the subspace $ \mathcal{C}(P_1)+ \mathcal{C}(P_2)$ of $\BH$ has codimension $2$ in $\BH$. Since we have
\[
T(\mathcal{C}(P_1)+ \mathcal{C}(P_2)) \subset \mathcal{C}(Q_1) + \mathcal{C}(Q_2), 
\]
the codimension of the subspace $ \mathcal{C}(Q_1)+ \mathcal{C}(Q_2)$ of $\BH$ is less than or equal to $2$.  
 Suppose that the rank of $Q_1$ is greater than $1$.  Choose linearly independent vectors $x_1, x_2 \in Q_1(\cH)$ with $x_1 \perp x_2$ and $x_3 \in Q_2(\cH) \setminus \{0\}$. The cosets of four operators $x_1 \otimes x_3$, $x_2 \otimes x_3$, $x_3 \otimes x_1$ and $x_3 \otimes x_2$ are linearly independent in $\BH/(\mathcal{C}(Q_1) + \mathcal{C}(Q_2))$. This contradicts the fact that  $ \mathcal{C}(Q_1)+ \mathcal{C}(Q_2)$ has codimension at most $2$. Thus the rank of $Q_1$ is $1$ and similarly the rank of $Q_2$ is $1$.
\end{proof}

We will use the following form of Uhlhorn's theorem.

\begin{theorem}[{Uhlhorn \cite{UHL}}]\label{uhlhorn-theorem}
Let $\cH$ be a complex Hilbert space with $\dim\cH\ge3$.
Suppose that $\Phi:\mathcal P_1(\cH)\to\mathcal P_1(\cH)$ is a bijection
such that
\[
P_1P_2=0\quad\Longleftrightarrow\quad
\Phi(P_1)\Phi(P_2)=0
\qquad(P_1,P_2\in\mathcal P_1(\cH)).
\]
Then there exists a unitary or antiunitary operator $U$ on $\cH$ such that
$\Phi(P)=UPU^*$ for every $P\in\mathcal P_1(\cH)$.
\end{theorem}

\begin{lemma}\label{UHL}
There exists a unitary or antiunitary operator $U$ on $\cH$ such that
\[
T(P)\in\Span\{I,UPU^*\}\qquad(P\in\PH).
\]
\end{lemma}

\begin{proof}
Lemma~\ref{scalar-preservation} gives $T(\bC I)=\bC I$.
Let $P\in\PH$. Since $T$ is injective, $T(P)$ is non-scalar. Lemma \ref{comm}, applied to $P$ and an orthogonal rank-one projection, shows there is $Q \in \PH$ such that $T(P) \in \Span\{Q, I\}$. We claim that this $Q$ is unique. Suppose that there are $Q_1, Q_2 \in \PH$ such that 
$T(P)\in \Span\{Q_1, I\} \cap \Span\{Q_2, I\}$. Then there are $\alpha_i, \beta_i \in \mathbb{C}$ such that
\[
\alpha_1 Q_1+\beta_1I=\alpha_2 Q_2+\beta_2I=T(P).
\]
Since $T(P)$ is non-scalar, $\alpha_1,\alpha_2\ne0$.
Since $\dim \cH \ge 3$, there is a non-zero vector $y \in (Q_1(\cH)+Q_2(\cH))^{\perp}$. We have $\beta_1y=\beta_2y$. Thus $\beta_1=\beta_2$, and hence $\alpha_1Q_1=\alpha_2Q_2$. Since $\alpha_1,\alpha_2\ne0$, the ranges of $Q_1$ and $Q_2$ are equal. Therefore $Q_1=Q_2$. 

Define $\Phi(P)=Q$. Lemma~\ref{comm} shows that $\Phi$ preserves
orthogonality. The same construction for $T^{-1}$ gives an
orthogonality-preserving map $\Psi:\PH\to\PH$.
The identities
\[
T(\Span\{I,P\})=\Span\{I,\Phi(P)\},\qquad
T^{-1}(\Span\{I,Q\})=\Span\{I,\Psi(Q)\}
\]
and the uniqueness just proved imply
$\Psi\circ\Phi=\operatorname{id}$ and
$\Phi\circ\Psi=\operatorname{id}$.
Thus $\Phi$ is bijective and preserves orthogonality in both directions.
Theorem~\ref{uhlhorn-theorem} gives a unitary or antiunitary operator
$U$ with $\Phi(P)=UPU^*$, as required.
\end{proof}

\begin{proof}[Proof of Theorem~\ref{main}: $(2)\Rightarrow(3)$]
Let $U$ be as in Lemma \ref{UHL}. Let $\mathbf e=(e_n) \in \mathcal{E}(\cH)$. Then $e_i \otimes e_i \in \mathcal{D}_{\mathbf e}$. Thus we have $T(e_i \otimes e_i) \in \mathcal{D}_{\varphi(\mathbf e)}$. By Lemma \ref{UHL} we have $T(e_i \otimes e_i) \in \Span \{U(e_i \otimes e_i)U^{*}, I\} =\Span \{Ue_i \otimes Ue_i, I\} $. Since $T(\bC I)=\bC I$ and $T$ is injective, $T(e_i\otimes e_i)$ is non-scalar, so the coefficient of $Ue_i\otimes Ue_i$ is nonzero. Thus $Ue_i \otimes Ue_i \in \mathcal{D}_{\varphi(\mathbf e)}$. Let $(f_n)=\varphi(\mathbf e)$.
Since $Ue_i\otimes Ue_i$ is a rank-one diagonal projection,
for each $i$, there are
$\pi(i)\in\mathcal I_N$ and $\lambda_i\in\mathbb T$ such
that
\[
Ue_i=\lambda_i f_{\pi(i)}.
\]
Since $(Ue_i)$ is an orthonormal basis, the map $\pi$ is a
permutation of $\mathcal I_N$. Therefore, $(Ue_n) \sim (f_n)$. 
Consequently, $\mathcal D_{(Ue_i)}=\mathcal D_{(f_i)}
=\mathcal D_{\varphi(\mathbf e)}$. 
We define a bijective additive map $\widetilde{T}$ on $\BH$ by $A \mapsto U^{*}T(A)U$. Then 
\[
\widetilde{T}(\mathcal{D}_{\mathbf e})=U^{*}T(\mathcal{D}_{\mathbf e})U=U^{*}\mathcal{D}_{(Ue_i)}U=\mathcal{D}_{\mathbf e}
\]
holds for any $\mathbf e \in \mathcal{E}(\cH)$. If $U$ is unitary, $\widetilde{T}$ is complex-linear; if $U$ is antiunitary, $\widetilde{T}$ is conjugate-linear. Applying Theorem \ref{AOST} and using the corresponding linearity, there exist $c \in \bC \setminus \{0\}$ and a functional $g$ on $\BH$, linear in the unitary case and conjugate-linear in the antiunitary case, such that
\[
\widetilde{T}(A)=
\begin{cases}
cA+g(A)I, & \text{if }U\text{ is unitary},\\[2pt]
cA^*+g(A)I, & \text{if }U\text{ is antiunitary},
\end{cases}
\qquad A\in \BH.
\]   
Since $\widetilde{T}$ is continuous, $g$ is bounded. Therefore,
\[
T(A)=
\begin{cases}
cUAU^*+g(A)I, & \text{if }U\text{ is unitary},\\[2pt]
\overline{c}UA^*U^*+\overline{g(A)}I, & \text{if }U\text{ is antiunitary}.
\end{cases}
\qquad A\in \BH.
\]
In the antiunitary case, the map $A\mapsto\overline{g(A)}$ is a bounded linear functional. Renaming the scalar and the functional gives the form stated in Theorem \ref{main}.
\end{proof}

\subsection*{Declaration on the Use of AI}
The authors used ChatGPT (OpenAI) to assist with
English-language editing. After using this tool, the authors reviewed and
edited the content as needed and take full responsibility for the content of
the publication.

\end{document}